\documentclass[11pt,a4paper]{amsart}

\usepackage[utf8]{inputenc}
\usepackage[T1]{fontenc}

\usepackage{amssymb}
\usepackage{amsmath}
\usepackage{amsfonts}
\usepackage{amssymb}
\usepackage{amsthm}
\usepackage{enumerate}
\usepackage{hyperref}

\newtheorem{theorem}{Theorem}[section]
\newtheorem{lemma}[theorem]{Lemma}
\newtheorem{proposition}[theorem]{Proposition}

\theoremstyle{definition}

\newtheorem{remark}[theorem]{Remark}

\newcommand{\R}{\mathbb{R}}
\newcommand{\N}{\mathbb{N}}

\newcommand{\Z}{\mathbb{Z}}

\newcommand{\cA}{{\mathcal A}}   
   
\newcommand{\cC}{{\mathcal C}}

\newcommand{\cH}{{\mathcal H}}

\newcommand{\cK}{{\mathcal K}}

\newcommand{\cO}{{\mathcal O}}

\newcommand{\cS}{{\mathcal S}}

\newcommand{\cW}{{\mathcal W}}

\newcommand{\mR}{{\mathbf R}}
\newcommand{\mK}{{\mathbf K}}

\newcommand{\weakto}{\rightharpoonup}

\newcommand{\eps}{\varepsilon}

\def\eps{\varepsilon}

\DeclareMathOperator{\supp}{supp}

\begin{document}

\title{On the periodic and asymptotically periodic nonlinear Helmholtz equation}
\author{Gilles Ev\'equoz}
\address{ Institut f\"ur Mathematik, Johann Wolfgang Goethe-Universit\"at,
Robert-Mayer-Str. 10, 60054 Frankfurt am Main, Germany}
\email{evequoz@math.uni-frankfurt.de}

\begin{abstract}
In the first part of this paper, the existence of infinitely many $L^p$-standing wave solutions for the nonlinear Helmholtz equation
$$
-\Delta u -\lambda u=Q(x)|u|^{p-2}u\quad\text{ in }\R^N
$$
is proven for $N\geq 2$ and $\lambda>0$, under the assumption that $Q$ be a nonnegative, periodic and bounded function and the exponent $p$ lies in the Helmholtz subcritical range.
In a second part, the existence of a nontrivial solution is shown in the case where the coefficient $Q$ is only asymptotically periodic.
\end{abstract}

\maketitle

\section{Introduction}
In this paper, we consider for $N\geq 2$ the semilinear equation
\begin{equation}\label{eq:33b}
- \Delta u -  \lambda u = Q(x)|u|^{p-2}u,\qquad x\in \R^N,
\end{equation}
where $\lambda>0$, $Q$ is a bounded and nonnegative function, and the exponent $p$ lies in the subcritical range 
$2_\ast<p<2^\ast$, where
$$
2_\ast:=\frac{2(N+1)}{N-1}\quad\text{and}\quad 
2^\ast=\left\{\begin{array}{ll} \frac{2N}{N-2} & \text{ if }N\geq 3 \\ \infty & \text{ if }N=2. \end{array}\right.
$$
We study existence of real-valued solutions $u$: $\R^N$ $\to$ $\R$ that decay to zero at infinity. 
Such solutions correspond, via the Ansatz
$$
\psi(t,x):=e^{i\sqrt{\lambda}t}u(x),
$$
to weakly spatially decaying, standing wave solutions of the Nonlinear Wave Equation
$$
\partial_{tt}\psi-\Delta\psi = Q(x)|\psi|^2\psi,\qquad t\in\R,\ x\in\R^N.
$$
Under general assumptions on $Q$, the problem \eqref{eq:33b} cannot be
handled using direct variational methods, since its solutions (if any) are not expected 
to decay faster than $O(|x|^{-\frac{N-1}{2}})$ at infinity and will therefore 
not belong to the space $L^2(\R^N)$. 
Recently, a dual method has been set up, which allows to study this problem variationally.
Using it, the existence of nontrivial solutions lying in $L^q(\R^N)$ for all $q\geq p$ 
and admitting an expansion of the form
$$
\lim\limits_{R\to\infty}\frac1R\int_{B_R}\left|u(x)+2\left(\frac{2\pi}{\sqrt{\lambda} |x|}\right)^{\frac{N-1}{2}}
\text{Re}\left[e^{i\sqrt{\lambda}|x|-\frac{i(N-1)\pi}{4}} g_u\left(\frac{x}{|x|}\right)\right]\right|^2\, dx=0,
$$
where $g_u(\xi)=-\frac{i}{4}\left(\frac{\lambda}{2\pi}\right)^{\frac{N-2}{2}}\mathcal{F}\left(Q|u|^{p-2}u\right)(\sqrt{\lambda}\xi)$, $\xi\in S^{N-1}$,
was proven (see \cite{evequoz-weth-dual,e-helm-2d}). Here, $\mathcal{F}$ denotes the Fourier transform. More precisely, infinitely many nontrivial bounded $W^{2,p}(\R^N)$-solutions were obtained under the assumption 
$\lim\limits_{|x|\to\infty}Q(x)= 0$ (see also \cite{evequoz:2015-1} where more general nonlinearities were considered). 
For periodic $Q$, only the existence of a nontrivial solution (pair) was proven, and one of the main goals of the present paper is to show that \eqref{eq:33b} in fact possesses infinitely many geometrically distinct solutions in $W^{2,p}(\R^N)$. 

In the case where $\lambda<0$, or more generally when $-\lambda$ is replaced by a bounded 
and periodic function $V$ satisfying $\inf V>0$, results giving the existence of infinitely many 
solutions for \eqref{eq:33b} go back to the work of Coti Zelati and Rabinowitz \cite{CR92}. 
Shortly after, Alama and Li \cite{alama-li92b} using a dual method, and Kryszewski and Szulkin
\cite{kryszewski-szulkin98} using a linking argument and a new degree theory, 
extended this result to the case where $0$ lies in a spectral gap of the 
Schr\"odinger operator $-\Delta+V(x)$. More recently, multiplicity results were given for the periodic Schr\"odinger equation 
with more general nonlinearities (see e.g. \cite{ackermann06,ackermann-weth05,chen08,ding-lee06,szulkin-weth09,squassina-szulkin}
 and the references therein).

In the present paper, we show that the proof scheme developed by Szulkin and Weth in \cite[Theorem 1.2]{szulkin-weth09} for the 
periodic Schr\"odinger equation and used recently by Squassina and Szulkin \cite{squassina-szulkin} in the case of a logarithmic
 nonlinearity, can be adapted to work in combination with the dual variational method.
In particular, using the framework of \cite{evequoz-weth-dual}, we obtain the following result.
\begin{theorem}\label{thm:infinite_periodic}
Let $2_\ast<p< 2^\ast$ and consider $Q\in L^\infty(\R^N)\backslash\{0\}$ 
nonnegative and $\Z^N$-periodic. Then \eqref{eq:33b} has infinitely many 
geometrically distinct pairs of strong solutions 
$\pm u_n\in W^{2,q}(\R^N)$, $p\leq q<\infty$.
\end{theorem}
The proof (see Section~\ref{sec:multiple}) is based on the fact that, assuming that only finitely many geometrically 
distinct solutions exist, one is able to show the discreteness of Palais-Smale sequences for the energy functional 
and then apply a deformation argument to get a contradiction. 
The main technical point, when passing to the dual method is to find suitable estimates to prove the analogue of 
\cite[Lemma 2.14]{szulkin-weth09}. This is done in Lemma \ref{lem:dist_PS_sequ} using the reverse H\"older inequality 
and the recently proven nonvanishing property for the subcritical Helmholtz equation \cite[Theorem 3.1]{evequoz-weth-dual} for $N\geq 3$ and \cite[Theorem 3.1]{e-helm-2d} for $N=2$.

Among one of the problems most extensively studied in the context of nonlinear Schr\"odin\-ger 
equation is the case of asymptotically periodic or asymptotically autonomous nonlinearity.
Starting with the work by Ding and Ni \cite{ding.ni:86} based on Lions' concentration-compactness principle, many authors studied this case 
(see e.g. \cite{bahri-lions97,cerami06, li-szulkin02} and the references therein). 
In Section \ref{sec:asympt_periodic} of the present paper, we give an analogous existence result for the nonlinear Helmholtz equation.
We namely study \eqref{eq:33b} in the case where the coefficient $Q$ satisfies $|Q(x)-Q_\infty(x)|\to 0$ as $|x|\to\infty$ for some periodic
function $Q_\infty$. Under assumptions close to those of Ding and Ni \cite{ding.ni:86}, we obtain the existence of a nontrivial solution 
(see Theorem ~\ref{thm:asympt_per}). 
Our proof uses the fibering method applied to the dual energy functional, and we show that the dual ground-state is attained 
under these assumptions.

Let us mention that the above results can be extended to slightly more general nonlinearities. Indeed, using a dual variational
method in Orlicz spaces (as discussed e.g. in \cite{evequoz:2015-1}) it can be shown that the problem
\begin{equation}\label{eqn:33b_generalized}
- \Delta u - \lambda u = \sum_{i=1}^mA(x)^{p_i}Q_i(x)|u|^{p_i-2}u,\qquad x\in \R^N
\end{equation}
with $\lambda>0$ has infinitely many geometrically distinct solutions under the assumptions:

(i) $2_\ast<p_1\leq p_2\leq \ldots\leq p_m<2^\ast$,

(ii) $Q_1, \ldots, Q_m\in L^\infty(\R^N)\backslash\{0\}$ are $\Z^N$-periodic with 
$0<\inf\limits_{\R^N}Q_i\leq \sup\limits_{\R^N}Q_i<\infty$
for all $1\leq i \leq m$, and

(iii) $A\in L^\infty(\R^N)\backslash\{0\}$ is nonnegative and $\Z^N$-periodic.

\medskip

\noindent Theorem~\ref{thm:asympt_per} can also be extended to \eqref{eqn:33b_generalized}
under the assumptions (i), (ii) and

(iii') $A\in L^\infty(\R^N)\backslash\{0\}$ satisfies $|A(x)-A_\infty(x)|\to 0$, as $|x|\to\infty$,
and $A_\infty\leq A$ a.e. on $\R^N$ for some nonnegative, bounded and $\Z^N$-periodic
function $A_\infty$.

\section{Some preliminary results}
We begin by fixing some notation and recalling some of the results previously obtained on the semilinear problem \eqref{eq:33b}. For more details and proofs, we refer the reader to \cite{evequoz-weth-dual,e-helm-2d}.

Throughout this paper, we consider $N\geq 2$ and write
$2_\ast:=\frac{2(N+1)}{N-1}$ and $2^\ast:=\frac{2N}{N-2}$ if $N\geq 3$ and $2^\ast:=\infty$ if $N=2$.
For simplicity and without loss of generality, we shall take $\lambda=1$ in the Helmholtz 
equation \eqref{eq:33b} 
and therefore look at the problem
\begin{equation*}
- \Delta u - u = Q(x)|u|^{p-2}u,\qquad x\in \R^N.  
\end{equation*}
The outgoing radial fundamental solution of the Helmholtz operator $-\Delta-1$ in $\R^N$ is then given by
$$
\Phi(x)=\frac{i}{4}(2\pi|x|)^{N-2}H^{(1)}_{\frac{N-2}{2}}(|x|), \quad x\neq 0,
$$
where $H^{(1)}_{\nu}$ denotes the first Hankel function (or Bessel function of the third kind, see \cite{lebedev}),
and we denote by $\Psi$ its real part. As a consequence of results by Kenig, Ruiz and Sogge \cite{KRS87}, the mapping
$\mR$: $\cS(\R^N)$ $\to$ $\cS'(\R^N)$,
$$
\mR f:=\Psi\ast f=\text{Re}(\Phi)\ast f, \quad f\in\cS(\R^N),
$$
where $\cS(\R^N)$ and $\cS'(\R^N)$ denote respectively the Schwartz space and the space of tempered distributions,
can be extended in a unique way to a continuous operator $\mR$: $L^{p'}(\R^N)$ $\to$ $L^p(\R^N)$
for each $2_\ast\leq p\leq 2^\ast$ (resp. $p<\infty$ in the case $N=2$). 
Here and in the following, $p'=\frac{p}{p-1}$ denotes the conjugate exponent to $p$.
Moreover, for every $f\in L^{p'}(\R^N)$, the function $u=\mR f$ belongs to $W^{2,p'}_{\text{loc}}(\R^N)$ and 
solves the Helmholtz equation $-\Delta u-u=f$ in the strong sense.

Let $Q\in L^\infty(\R^N)$ be a nonnegative function, $Q\not\equiv 0$ and define for $2_\ast\leq p\leq 2^\ast$ with $p<\infty$,
the Birman-Schwinger type operator $\mK$: $L^{p'}(\R^N)$ $\to$ $L^p(\R^N)$ by setting
$$
\mK v:=Q^\frac1p \mR(Q^\frac1p v), \quad v\in L^{p'}(\R^N).
$$
This operator is symmetric, i.e., $\int_{\R^N}v\mK(w)\, dx=\int_{\R^N}w\mK(v)\, dx$ holds for $v, w\in L^{p'}(\R^N)$.
Moreover, when $p<2^\ast$ it is locally compact in the sense that for every bounded and measurable set $B\subset\R^N$,
the operator $1_B\mK$: $L^{p'}(\R^N)$ $\to$ $L^p(\R^N)$ is compact (see \cite[Lemma 4.1]{evequoz-weth-dual}). 
Here and in the sequel $1_B$ will stand for the characteristic function of the set $B$.

Consider the energy functional $J$: $L^{p'}(\R^N)$ $\to$ $\R$
\begin{equation}\label{eqn:energy_funct}
\begin{aligned}
J(v)&=\frac{1}{p'}\int_{\R^N}|v|^{p'}\, dx - \frac12 \int_{\R^N}Q(x)^\frac1pv(x) \mR(Q^\frac1p v)(x)\, dx
    =\frac1{p'}\|v\|_{p'}^{p'}-\frac12 \int_{\R^N}v\mK(v)\, dx.
\end{aligned}
\end{equation}
We note that $J$ is of class $C^1$ with
\begin{equation}\label{eqn:gradient}
J'(v)w=\int_{\R^N}\Bigl(|v|^{p'-2}v-\mK(v)\Bigr)w\, dx\quad\text{for all }v,w\in L^{p'}(\R^N).
\end{equation}
Furthermore, every critical point of $J$ corresponds to a solution of \eqref{eq:33b}.
More precisely, $v\in L^{p'}(\R^N)$ satisfies $J'(v)=0$ if and only if it solves the integral equation
$$
|v|^{p'-2}v=Q^\frac1p \mR(Q^\frac1p v).
$$
Setting $u=\mR(Q^\frac1p v)\in L^p(\R^N)$, it follows that
$$
u= \mR(Q|u|^{p-2}u),
$$
and $u\in W^{2,q}(\R^N)$, $q\geq p$, is therefore a strong solution of \eqref{eq:33b} (see
\cite[Lemma 4.3 and Theorem 4.4]{evequoz-weth-dual} and \cite[Theorem 1.1]{e-helm-2d}
concerning the regularity and asymptotic behavior of $u$).
We note also that $u=0$ if and only if $v=0$.

The geometry of the functional $J$ is of mountain pass type 
(cf. the original paper by Ambrosetti and Rabinowitz \cite{ambrosetti-rabinowitz73}):
\begin{lemma}\label{lem:MP}
\begin{itemize}
	\item[(i)] There exist $\alpha>0$ and $0<\rho<1$ such that $J(v)\geq\alpha>0$ for all 
	$v\in L^{p'}(\R^N)$ with $\|v\|_{p'}=\rho$.
	\item[(ii)] For every $m \in \N$, there exists an $m$-dimensional subspace $\cW_m \subset \cC_c^\infty(\R^N)$ 
	and some $R=R(\cW_m)>0$ such that $J(v) \le 0$ for every $v\in \cW_m$ with $\|v\|_{p'} \ge R$.
\end{itemize}
\end{lemma}
The proof of these results can be found in \cite[Lemma 4.2 and Lemma 5.1]{evequoz-weth-dual}
and \cite[Proof of Theorem 1.3(a)]{e-helm-2d}.

\medskip

We mention a useful property of the Palais-Smale sequences for $J$ which can be deduced
from the results in \cite{evequoz-weth-dual}. For the reader's convenience, we give a short proof.
\begin{lemma}\label{lem:PS_sequences}
Suppose $2_\ast\leq p<2^\ast$, and let $(v_n)_n\subset L^{p'}(\R^N)$ be a Palais-Smale sequence for $J$. Then $(v_n)_n$ 
is bounded in $L^{p'}(\R^N)$ and (up to a subsequence) there exists $v\in L^{p'}(\R^N)$ such that $J'(v)=0$, 
$v_n\weakto v$ weakly in $L^{p'}(\R^N)$ and $J(v)\leq\liminf\limits_{n\to\infty} J(v_n)$.
Moreover, for every bounded and measurable set $B\subset\R^N$ there holds $1_Bv_n\to 1_Bv$ strongly in $L^{p'}(\R^N)$ 
and $1_B|v_n|^{p'-2}v_n\to 1_B|v|^{p'-2}v$ strongly in $L^p(\R^N)$.
\end{lemma}
\begin{proof}
To prove the boundedness of the sequence, we choose $C>0$ such that $J(v_n)\leq \frac{C}{2}$ and $\|J'(v_n)\|_\ast\leq C$ 
for all $n\in\N$. If $\|v_n\|_{p'}\geq 1$ we can write
$$
C\geq \frac{J(v_n)}{\|v_n\|_{p'}}-\frac12J'(v_n)\frac{v_n}{\|v_n\|_{p'}}=\left(\frac1{p'}-\frac12\right)\|v_n\|_{p'}^{p'-1}.
$$
As a consequence, we find for all $n$, $\|v_n\|_{p'}^{p'-1}\leq\max\{1,\left(\frac1{p'}-\frac12\right)^{-1}C\}$, since $1<p'<2$,
thus showing that $(v_n)_n$ is bounded in $L^{p'}(\R^N)$. Hence, going if necessary to a subsequence, we
may assume that $v_n\weakto v$ weakly in $L^{p'}(\R^N)$ for some $v\in L^{p'}(\R^N)$. Let $B\subset\R^N$ be
bounded and measurable. For $m, n\in\N$ and $\varphi\in L^{p'}(\R^N)$ we find that
$$
\int_B\left|(|v_n|^{p'-2}v_n-|v_m|^{p'-2}v_m)\varphi\right|\, dx\leq \|J'(v_n)-J'(v_m)\|_\ast \|\varphi\|_{p'}
+\|1_B\mK(v_n-v_m)\|_p \|\varphi\|_{p'}.
$$
Since $1_B\mK$ is compact, we infer from the above estimate
that $(1_B|v_n|^{p'-2}v_n)_n$ is a Cauchy sequence in $L^p(\R^N)$. Hence, there exists $z\in L^p(\R^N)$
such that $1_B|v_n|^{p'-2}v_n\to z$ strongly in $L^p(\R^N)$ and therefore
\begin{align*}
\int_{\R^N}\left|1_Bv_n-|z|^{p-2}z\right|^{p'}\, dx 
&\leq C\int_{\R^N}\left(|v_n|^{p'-1}+|z|\right)^{(p-2)p'}\left|1_B|v_n|^{p'-2}v_n-z\right|^{p'}\, dx\\
&\leq C\left(\|v_n\|^{p'}_{p'}+\|z\|_p^{p'}\right)\|1_B|v_n|^{p'-1}v_n-z\|_{p'}.
\end{align*}
As a consequence, we obtain $1_Bv_n\to |z|^{p-2}z$ as $n\to\infty$, strongly in $L^{p'}(\R^N)$. 
Moreover, by uniqueness of the weak limit,
there holds $z=1_B|v|^{p'-2}v$, and this gives the desired strong local convergence.

As a consequence, we find that for every $\varphi\in\cC^\infty_c(\R^N)$,
\begin{align*}
J'(v)\varphi&=\int_{\supp\varphi}|v|^{p'-2}v\varphi\, dx -\int_{\R^N}\varphi\mK(v)\, dx\\
&=\lim_{n\to\infty}\int_{\supp\varphi}|v_n|^{p'-2}v_n\varphi\, dx -\int_{\R^N}\varphi\mK(v_n)\, dx
=\lim_{n\to\infty}J'(v_n)\varphi=0,
\end{align*}
since $v_n\weakto v$. Hence $J'(v)=0$, and since the norm $\|\cdot\|_{p'}$ is weakly
lower sequentially continuous, we infer that
\begin{align*}
J(v)&= J(v)-\frac12J'(v)v=\left(\frac1{p'}-\frac12\right)\int_{\R^N}|v|^{p'}\, dx \\
&\leq\liminf_{n\to\infty}\left(\frac1{p'}-\frac12\right)\int_{\R^N}|v_n|^{p'}\, dx 
= \liminf_{n\to\infty}J(v_n)-\frac12J'(v_n)v_n=\liminf_{n\to\infty}J(v_n),
\end{align*}
and this concludes the proof.
\end{proof}

In the case where $2_\ast<p<2^\ast$ and the coefficient $Q$ is $\Z^N$-periodic, we proved
in \cite{evequoz-weth-dual} the existence of a nontrivial critical point for $J$. 
A crucial ingredient in our proof was the following nonvanishing property
(see \cite[Theorem 3.1]{evequoz-weth-dual} and \cite[Theorem 3.1]{e-helm-2d}).

\begin{theorem}\label{thm:nonvanishing}
Let $2_\ast< p<2^\ast$ and consider a bounded sequence $(v_n)_n\subset L^{p'}(\R^N)$ satisfying 
$\limsup \limits_{n\to\infty}\left|\int_{\R^N}v_n\mR v_n\, dx\right|>0.$ 
Then there exists $R>0$, $\zeta>0$ and a sequence
$(x_n)_n\subset\R^N$ such that, up to a subsequence, 
\begin{equation}\label{eqn:liminf2}
\int_{B_R(x_n)}|v_n|^{p'}\, dx \geq \zeta\quad\text{for all }n.
\end{equation}
\end{theorem}

\section{Existence of infinitely many solutions in the periodic setting}\label{sec:multiple}
In this section, we prove the multiplicity result, Theorem \ref{thm:infinite_periodic}, 
announced in the introduction.
We assume throughout that $Q$ is bounded, nonnegative, $\not\equiv 0$ and $\Z^N$-periodic, and
let $\cK:=\{u\in L^{p'}(\R^N)\, :\, J'(u)=0\}$. Moreover, we denote by $\cA$ a symmetric subset of $\cK$ (i.e.
$\cA=-\cA$) which contains exactly one element from each orbit $\cO(w)=\{w(\cdot-y)\, :\, y\in\Z^N\}$. 
From now on, let us suppose by contradiction that 
\begin{center} 
{\em the set $\cA$ is finite.}
\end{center}

\noindent We follow below the ideas of the proof of \cite[Theorem 1.2]{szulkin-weth09} 
and \cite[Theorem 1.1]{squassina-szulkin}.

\begin{lemma}\label{lem:dist_crit_pts}
$\kappa:=\inf\{\|v-w\|_{p'}\, :\, v, w\in\cK, v\neq w\}>0$.
\end{lemma}
\begin{proof}
Let $(v_n)_n, (w_n)_n\subset\cK$ be sequences such that $v_n\neq w_n$ for all $n$ and $\|v_n-w_n\|_{p'}\to \kappa$,
as $n\to\infty$. Then there are sequences $(y_n)_n, (z_n)_n\subset\Z^N$ for which $\tilde{v}_n:=v_n(\cdot+y_n)\in \cA$
and $\tilde{w}_n:=w_n(\cdot+z_n)\in \cA$ for all $n$.
Since $\cA$ is finite, there holds (up to a subsequence) $\tilde{v}_n=\tilde{v}\in\cA$, $\tilde{w}_n=\tilde{w}\in\cA$ for all $n$,
and either $y_n-z_n=y_0\in\Z^N$ for all $n$, or $|y_n-z_n|\to\infty$.

In the case where $y_n-z_n=y_0\in\Z^N$ for all $n$, we find
$$
0<\|\tilde{v}(\cdot-y_0)-\tilde{w}\|_{p'}
=\|\tilde{v}_n(\cdot-y_n)-\tilde{w}_n(\cdot-z_n)\|_{p'}=\|v_n-w_n\|_{p'}=\kappa,
$$
since the orbits $\cO(\tilde{v})$ and $\cO(\tilde{w})$ are distinct.
In the second case, we may take without loss of generality $\tilde{v}\neq 0$ and we remark that
$\tilde{w}(\cdot+y_n-z_n)\weakto 0$ in $L^{p'}(\R^N)$. This gives
$$
\kappa=\liminf_{n\to\infty}\|\tilde{v}_n(\cdot-y_n)-\tilde{w}_n(\cdot-z_n)\|_{p'}
=\liminf_{n\to\infty}\|\tilde{v}-\tilde{w}(\cdot+y_n-z_n)\|_{p'}
\geq \|\tilde{v}\|_{p'}>0,
$$
and thus concludes the proof.
\end{proof}

\begin{lemma}\label{lem:dist_PS_sequ}
Let $2_\ast<p<2^\ast$ and consider two Palais-Smale sequences 
$(v_n)_n, (w_n)_n\subset L^{p'}(\R^N)$ for $J$. 
Then, either $\|v_n-w_n\|_{p'}\to 0$,  as $n\to\infty$, or $\limsup\limits_{n\to\infty}\|v_n-w_n\|_{p'}\geq \kappa$,
where $\kappa$ is given by Lemma~\ref{lem:dist_crit_pts}.
\end{lemma}
\begin{proof}
First note that the sequences $(v_n)_n, (w_n)_n$ are bounded.
We distinguish two cases.

{\bf Case 1}: If $\int_{\R^N}(v_n-w_n)\mK (v_n-w_n)\, dx\to 0$, as $n\to\infty$, it follows that
\begin{align}
\int_{\R^N}(|v_n|^{p'-2}v_n-&|w_n|^{p'-2}w_n)(v_n-w_n)\, dx\label{eqn:difference}\\
& =(J'(v_n)-J'(w_n))(v_n-w_n)+\int_{\R^N}(v_n-w_n)\mK (v_n-w_n)\, dx\to 0,\nonumber
\end{align}
as $n\to\infty$. Moreover, since $1<p'<2$, there holds for every $a, b>0$:
\begin{equation*}
(a^{p'-1}-b^{p'-1})(a-b)=(p'-1)(a-b)\int_b^a t^{p'-2}\, dt \geq (p'-1)(a-b)^2(a+b)^{p'-2}.
\end{equation*}
Using the Reverse H\"older inequality (see \cite[Theorem 2.12]{adams}) we obtain
\begin{align*}
\int_{\R^N}(|v_n|^{p'-2}v_n&-|w_n|^{p'-2}w_n)(v_n-w_n)\, dx
\geq (p'-1)\int_{\R^N}(v_n-w_n)^2(|v_n|+|w_n|)^{p'-2}\, dx\\
&\geq (p'-1) \left(\int_{\R^N}|v_n-w_n|^{p'}\, dx\right)^{\frac{2}{p'}}\left(\int_{\R^N}(|v_n|+|w_n|)^{p'}\, dx\right)^{1-\frac{2}{p'}}.
\end{align*}
Using \eqref{eqn:difference} and the boundedness of $(v_n)_n$ and $(w_n)_n$, we deduce that $\|v_n-w_n\|_{p'}\to 0$
as $n\to\infty$.

{\bf Case 2}: If $\limsup\limits_{n\to\infty}\left|\int_{\R^N}(v_n-w_n)\mK (v_n-w_n)\, dx\right|>0$, 
the nonvanishing property (see Theorem~\ref{thm:nonvanishing}) gives the existence of 
$R, \zeta>0$ and a sequence $(x_n)_n\subset\R^N$ such that (up to a subsequence)
\begin{equation}\label{eqn:mass_concentr}
\int_{B_R(x_n)}|v_n-w_n|^{p'}\, dx\geq \zeta\quad\text{for all }n.
\end{equation}
We may even assume (making $R$ larger if necessary) that $x_n\in\Z^N$ for all $n$. 
Setting $\tilde{v}_n=v_n(\cdot+x_n)$
and $\tilde{w}_n=w_n(\cdot+x_n)$ for each $n$, we find, using the $\Z^N$-translation 
invariance of $J$, that $(\tilde{v}_n)_n$ and 
$(\tilde{w}_n)_n$ are bounded Palais-Smale sequences for $J$. From Lemma~\ref{lem:PS_sequences}, 
there exist $\tilde{v}, \tilde{w}\in\cK$ satisfying
$\tilde{v}_n\weakto \tilde{v}$ and $\tilde{w}_n\weakto \tilde{w}$ weakly in $L^{p'}(\R^N)$
as well as $1_{B_R}\tilde{v}_n\to1_{B_R}\tilde{v}$ and $1_{B_R}\tilde{w}_n\to1_{B_R}\tilde{w}$ strongly in $L^{p'}(\R^N)$. 
The property \eqref{eqn:mass_concentr} then implies
$\tilde{v}\neq\tilde{w}$, and we conclude that
$$
\limsup_{n\to\infty}\|v_n-w_n\|_{p'}\geq\liminf_{n\to\infty}\|\tilde{v}_n-\tilde{w}_n\|_{p'}
\geq \|\tilde{v}-\tilde{w}\|_{p'}\geq\kappa.
$$
\end{proof}

Let $H$: $L^{p'}(\R^N)\backslash\cK$ $\to$ $L^{p'}(\R^N)$ be a locally Lipschitz continuous and $\Z_2$-equivariant pseudo-gradient 
vector field for $J$, i.e., 
$$\|H(v)\|_{p'}<2\min\{\|J'(v)\|,1\}\quad\text{ and }\quad J'(v)H(v)>\min\{\|J'(v)\|,1\}\|J'(v)\|,\quad v\in L^p(\R^N)$$
(see~\cite[Definition II.3.1]{struwe}). We consider for $v\in L^{p'}(\R^N)\backslash\cK$ the flow $\eta$ given by
\begin{equation}
\left\{\begin{array}{rcl} \frac{\partial}{\partial t}\eta(t,v)&=&-H(\eta(t,v)),\\ \eta(0,v)&=&v, \end{array}\right.
\end{equation}
and denote by $(T^-(v),T^+(v))$ the maximal existence time interval for the 
trajectory $t\mapsto \eta(t,v)$.
\begin{lemma}\label{lem:max_exist_time}
For every $v\in L^{p'}(\R^N)\backslash\cK$, the following alternative holds: 
either $\lim\limits_{t\to T^+(v)}\eta(t,v)$ exists and is a critical point of $J$, 
or $\lim\limits_{t\to T^+(v)}J(\eta(t,v))=-\infty$. In the latter case, $T^+(v)=\infty$.
\end{lemma}
\begin{proof}
Let $v\in L^{p'}(\R^N)$ and let us first remark that $J$ is bounded on bounded sets and that it is strictly decreasing 
along trajectories of $\eta$. 
In particular, if $t\mapsto J(\eta(t,v))$ becomes unbounded as $t\to T^+(v)$, then
$\|\eta(t,v)\|_{p'}\to \infty$ as $t\to T^+(v)$ and the boundedness of $H$ implies $T^+(v)=\infty$ and 
$\lim\limits_{t\to T^+(v)}J(\eta(t,v))=-\infty$. 

Let us now assume that $\ell:=\inf\{J(\eta(t,v))\, :\, 0\leq t<T^+(v)\}>-\infty$, and observe that
$\ell=\lim\limits_{t\to T^+(\eta)}J(\eta(t,v))$ holds, since $J$ is strictly decreasing along the flow.
If $T^+(v)<\infty$, we can write for $0\leq s<t<T^+(v)$:
\begin{align*}
\|\eta(t,v)-\eta(s,v)\|_{p'}&\leq \int_s^t\|H(\eta(\tau,v))\|_{p'}\, d\tau < 2(t-s).
\end{align*}
Hence the limit $\lim\limits_{t\to T^+(v)}\eta(t,v)$ exists and is a critical point, since otherwise 
the trajectory could be continued beyond $T^+(v)$.
It remains to consider the case where $T^+(v)=\infty$. In order to prove the existence of 
$\lim\limits_{t\to\infty}\eta(t,v)$ in this case, we show 
that for every $\eps>0$ there exists $t_\eps>0$ such that 
$\|\eta(t_\eps,v)-\eta(t,v)\|_{p'}<\eps$ for all $t\geq t_\eps$. 
Assume by contradiction that this property does not hold. Then for some
$0<\eps<\frac{\kappa}{2}$ there is an increasing sequence 
$(t_n)_n\subset(0,\infty)$ satisfying $\lim\limits_{n\to\infty}t_n=\infty$ and 
$\|\eta(t_n,v)-\eta(t_{n+1},v)\|_{p'}=\eps$ for all $n$. 
Choosing the smallest $t_n^1\in(t_n,t_{n+1})$ such that 
$\|\eta(t_n,v)-\eta(t_n^1,v)\|_{p'}=\frac{\eps}{3}$ and setting 
$\kappa_n:=\min\limits_{s\in[t_n,t_n^1]}\|J'(\eta(s,v))\|$, 
we obtain from the properties of the pseudo-gradient field that
\begin{align*}
\frac{\eps}{3}&=\|\eta(t_n,v)-\eta(t_n^1,v)\|_{p'}\leq \int_{t_n}^{t_n^1}\| H(\eta(s,v))\|_{p'}\, ds\\
&\leq \frac{2}{\kappa_n}\int_{t_n}^{t_n^1}J'(\eta(s,v))H(\eta(s,v))\, ds
= \frac{2}{\kappa_n}\left(J(\eta(t_n,v))-J(\eta(t_n^1,v))\right).
\end{align*}
Since $J(\eta(t_n,v))-J(\eta(t_n^1,v))\to \ell-\ell=0$ as $n\to\infty$, 
we obtain $\lim\limits_{n\to\infty}\kappa_n=0$. Therefore,
we find $s_n^1\in[t_n,t_n^1]$ such that $J'(\eta(s_n^1,v))\to 0$ as $n\to\infty$. 
Choosing the largest $t_n^2\in(t_n^1,t_{n+1})$ for which
$\|\eta(t_n^2,v)-\eta(t_{n+1},v)\|_{p'}=\frac{\eps}{3}$, we find similarly some 
$s_n^2\in [t_n^2,t_{n+1}]$ such that $J'(\eta(s_n^2,v))\to 0$
as $n\to\infty$. Hence, $(\eta(s_n^1,v))_n$ and $(\eta(s_n^2, v))_n$ are 
Palais-Smale sequences for $J$ and satisfy for all $n$:
$$
\frac{\eps}{3}\leq \|\eta(s_n^1,v)-\eta(s_n^2,v)\|_{p'}\leq 2\eps<\kappa,
$$
which contradicts Lemma~\ref{lem:dist_PS_sequ}. 
Thus, $v^\ast:=\lim\limits_{t\to\infty}\eta(t,v)$ exists and the estimate
$$
\int_0^\infty \min\{\|J'(\eta(s,v))\|^2,1\} \, ds\leq J(v)-J(v^\ast)
$$ 
implies that it is a critical point.
\end{proof}

We note that, since $\cA$ is finite, there is for each fixed $d>0$ some $\eps_0>0$ such that
$J^{d+2\eps_0}_{d-2\eps_0}\cap\cK=\cK_d$. Here and in the sequel, we use the following notation for
super- and sublevel sets of $J$: for $a,b\in \R$, we let
$J_a:=\{u\in L^{p'}(\R^N)\, :\, J(u)\geq a\}$, $J^b:=\{u\in L^{p'}(\R^N)\, :\, J(u)\leq b\}$ and $J_a^b:=J_a\cap J^b$.
In addition, we set for $d\in\R$, $\cK_d:=\{u\in\cK\, :\, J(u)=d\}$.
For given $\delta>0$ we shall denote, in the following, the open $\delta$-neighborhood of a set $B\subset L^{p'}(\R^N)$ by
$U_\delta(B)$. 

\begin{lemma}\label{lem:eps_0}
Let $d>0$ and consider $\eps_0>0$ such that $J^{d+2\eps_0}_{d-2\eps_0}\cap\cK=\cK_d$. 
For every $\delta>0$, there exists $0<\eps<\eps_0$
and a continuous $1$-parameter family of odd homeomorphisms of $L^{p'}(\R^N)$, 
$\{\tilde{\eta}(t,\cdot)\}_{t\in\R}$, such that
\begin{itemize}
 \item[(i)] $t\mapsto J(\tilde{\eta}(t,v))$ is nonincreasing, $\forall$ $v\in L^{p'}(\R^N)$;
 \item[(ii)] For every $v\in J^{d+\eps}\backslash U_\delta(\cK_d)$, there exists $T\geq 0$ for which
 $J(\tilde{\eta}(T,v))\leq d-\eps$.
\end{itemize}
\end{lemma}

\begin{proof}
Taking without loss of generality $\delta<\kappa$, we first claim that
$$
\tau:=\inf\bigl\{\|J'(v)\|\, :\, v\in J^{d+\eps_0}_{d-\eps_0}\cap 
(U_\delta(\cK_d)\backslash U_{\frac{\delta}{2}}(\cK_d))\bigr\}>0.
$$
Assuming by contradiction that $\tau=0$ holds, we can find a sequence 
$(v_n)_n\subset U_\delta(\cK_d)\backslash U_{\frac{\delta}{2}}(\cK_d)$ 
such that $J(v_n)\in [d-2\eps_0,d+2\eps_0]$ for all $n$
and $J'(v_n)\to 0$ as $n\to\infty$. 
Using the $\Z^N$-invariance of $J$ and the finiteness of $\cA$, we find that (up to a subsequence) 
$(v_n)_n\subset U_\delta(\{w_0\})\backslash U_{\frac{\delta}{2}}(\{w_0\})$ for some $w_0\in \cK_d$. This implies that
$$
\frac{\delta}{2}\leq\limsup_{n\to\infty}\|v_n-w_0\|_{p'}\leq \delta<\kappa,
$$
which contradicts Lemma~\ref{lem:dist_PS_sequ}. Hence, $\tau>0$ must hold. 

\medskip

Without loss of generality we may assume $\tau<1$. Choosing 
$0<\eps<\min\bigl\{\eps_0,\frac{\delta\tau^2}{4}\bigr\}$,
we claim that the following properties hold for every $v\in J^{d+\eps}\backslash U_\delta(\cK_d)$:
\begin{eqnarray}
 \lim\limits_{t\to T^+(v)}J(\eta(t,v))<d-\eps \label{eqn:limit_below}\\
 \eta(t,v)\notin U_{\frac{\delta}{2}}(\cK_d)\cap J_{d-\eps}\quad\text{ for all }0\leq t<T^+(v).\label{eqn:eta_neighb}
\end{eqnarray}
Assuming by contradiction that at least one of the above properties does not hold, 
we find by Lemma~\ref{lem:max_exist_time} and the fact that $J^{d+2\eps_0}_{d-2\eps_0}\cap \cK=\cK_d$, some
$v\in J^{d+\eps}\backslash U_\delta(\cK_d)$ and some $0<t_0<T^+(v)$ such that
$\eta(t_0,v)\in U_{\frac{\delta}{2}}(\cK_d)\cap J_{d-\eps}$. Moreover, since $\delta<\kappa$,
there is, according to Lemma \ref{lem:dist_crit_pts}, $w\in \cK_d$ such that 
$\eta(t_0,v)\in U_{\frac{\delta}{2}}(\{w\})\cap J_{d-\eps}$.
Setting 
\begin{align*}
t_1:=\sup\{t\in[0,t_0)\, :\, \eta(t,v)\notin U_\delta(\{w\})\} \text{ and }
t_2:=\inf\{t\in(t_1,T^+(v))\, :\, \eta(t,v)\in U_\frac{\delta}{2}(\{w\})\}
\end{align*}
we obtain $t_2\leq t_0$ and
$$
\frac{\delta}{2}\leq\|\eta(t_1,v)-\eta(t_2,v)\|_{p'}\leq \int_{t_1}^{t_2}\|H(\eta(s,v))\|_{p'}\, ds\leq 2(t_2-t_1).
$$
Consequently,
\begin{align*}
J(\eta(t_0,v))&\leq J(\eta(t_2,v))=J(\eta(t_1,v))-\int_{t_1}^{t_2}J'(\eta(s,v))H(\eta(s,v))\, ds\\
&\leq (d+\eps)-\tau^2(t_2-t_1)<d-\eps,
\end{align*}
contradicting the fact that $\eta(t_0,v)\in J_{d-\eps}$. This contradiction proves \eqref{eqn:limit_below} and
\eqref{eqn:eta_neighb}.

\medskip

Let us now choose a locally Lipschitz continuous function $\chi$: $L^{p'}(\R^N)$ $\to$ $[0,1]$ 
such that $\chi(-v)=\chi(v)$ for all $v$, $\chi=1$ on $J^{d+\eps}_{d-\eps}\backslash U_{\frac{\delta}{2}}(\cK_d)$ 
and $\chi=0$ on $J_{d+2\eps_0}\cup J^{d-2\eps_0}\cup \left(\overline{U_{\frac{\delta}{4}}(\cK_d)}\cap J_{d-2\eps_0}\right)$.
For every $v\in L^{p'}(\R^N)$, the Cauchy problem
\begin{equation}
\left\{\begin{array}{rcl} \frac{\partial}{\partial t}\tilde{\eta}(t,v)&=&-\chi(\tilde{\eta}(t,v)H(\tilde{\eta}(t,v)),\\ 
\tilde{\eta}(0,v)&=&v, \end{array}\right.
\end{equation}
has a unique solution $\tilde{\eta}(\cdot,v)$ defined on $\R$, and the flow $\tilde{\eta}$ is continuous on 
$\R\times L^{p'}(\R^N)$. By the semigroup property and the fact that $\chi H$ is odd, 
$\tilde{\eta}(t,\cdot)$ is an odd homeomorphism for every $t\in\R$. Also,
since $\chi\geq 0$ we obtain from the properties of the pseudogradient field $H$, 
that $J$ is nonincreasing along the trajectories of $\tilde{\eta}$.

Let now $v\in J^{d+\eps}\backslash U_\delta(\cK_d)$ be chosen. 
If $J(v)<d-\eps$, then (ii) holds with $T=0$.
Otherwise, $v\in J^{d+\eps}_{d-\eps}\backslash U_\delta(\cK_d)$ and by \eqref{eqn:eta_neighb}, we obtain
$\chi(\eta(t,v))=1$ for all $0\leq t<T^+(v)$ such that $J(\eta(t,v))\geq d-\eps$. The uniqueness of the 
flow therefore implies $\tilde{\eta}(t,v)=\eta(t,v)$ for all such $t$, and using \eqref{eqn:limit_below},
we find $0\leq T<T^+(v)$ for which $J(\tilde{\eta}(T,v))=J(\eta(T,v))=d-\eps$. This shows 
that $\tilde{\eta}$ has the properties (i) and (ii) and concludes the proof.
\end{proof}

In order to obtain a contradiction to the assumption that $\cA$ is finite, 
we will need the following variant of Benci's pseudoindex (see \cite{bartolo-benci-fortunato83, benci82}).
Let $\Sigma$ denote the family of all compact and symmetric subsets of $L^{p'}(\R^N)$,
and consider $\rho>0$ as given by Lemma \ref{lem:MP}(i), i.e., small enough that for some $\alpha>0$, 
$J(v)\geq \alpha$ for all $v\in S_\rho(0)=\{v\in L^{p'}(\R^N)\, :\, \|v\|_{p'}=\rho\}$.
For $A\in\Sigma$, set
$$
i^\ast(A):=\min\bigl\{\gamma(h(A)\cap S_\rho(0))\, :\, h\in\cH\bigr\},
$$
where $\gamma$ denotes the Krasnoselskii genus and
$\cH:=\bigl\{h: L^{p'}(\R^N)\to L^{p'}(\R^N)$ odd homeomorphism with
$J(h(v))\leq J(v)\text{ for all }v\in L^{p'}(\R^N)\bigr\}$. 

We note that according to \cite[Proposition 1.6]{benci82}, $i^\ast$ is a pseudoindex 
in the sense of \cite[Definition 1.2]{benci82}. Moreover, using Lemma \ref{lem:MP},
we find that sets of arbitrarily large pseudoindex exist. Indeed, arguing as in 
\cite[Lemma 2.16]{squassina-szulkin} (iv), we find that for every $m$, 
$i^\ast(\cW_m\cap \overline{B_R}(0))\geq m$,
where $\cW_m$ and $R=R(\cW_m)$ are as in Lemma~\ref{lem:MP}(ii).

We can now give the proof of the multiplicity result.
\begin{proof}[Proof of Theorem~\ref{thm:infinite_periodic}]
For $k\in\N$, let
$$
d_k:=\inf_{\substack{A\in\Sigma\\ i^\ast(A)\geq k}}\sup J(A).
$$
Since there exist sets of arbitrarily large pseudoindex, $d_k$ is well-defined for all $k$. 
Moreover our choice of
$\rho$ gives $d_k\geq \alpha$, $\forall$ $k\geq 1$. 
We shall prove that for every $k$,
\begin{equation}\label{eqn:claim0}
\cK_{d_k}\neq\varnothing\quad\text{and}\quad d_k<d_{k+1}.
\end{equation}

Let $k\in\N$ and consider $d=d_k$. First remark that by Lemma~\ref{lem:dist_crit_pts}, 
$\cK$ is a countable discrete set and therefore $\gamma(\cK_d)=0$ or $1$. 
Taking $\delta>0$ and $U:=U_\delta(\cK_d)$ such that $\gamma(\overline{U})=\gamma(\cK_d)$, 
we consider corresponding $\eps>0$ and $\{\tilde{\eta}(t,\cdot)\}_{t\geq 0}$ given by Lemma \ref{lem:eps_0}.

Let $A\in\Sigma$ be chosen in such a way that $i^\ast(A)\geq k$ and $\sup J(A)\leq d+\eps$. 
Since $A$ is compact, the property (ii) of the family 
$\{\tilde{\eta}(t,\cdot)\}_{t\geq 0}$, implies
the existence of some common $T\geq 0$ such that 
$J(\widetilde{\eta}(T,v))\leq d-\eps$ for all $v\in A\backslash U$. 
From the properties of the pseudoindex 
(see \cite[Definition 1.2]{benci82}) and since $\tilde{\eta}(T,\cdot)\in\cH$, we obtain 
$$
k\leq i^\ast(A)\leq i^\ast(A\backslash U)+\gamma(\overline{U})
\leq i^\ast\bigl(\widetilde{\eta}(T,A\backslash U)\bigr)+\gamma(\cK_d)
\leq k-1+\gamma(\cK_d).
$$
Hence $\gamma(\cK_d)\geq 1$, and therefore $\cK_d\neq\varnothing$. 
Moreover, if $d_k=d_{k+1}$ would hold, then we could choose $A$ such that 
$i^\ast(A)\geq k+1$ in the above argument and this would give $\gamma(\cK_d)\geq 2$, 
contradicting the fact that $\gamma(\cK)\leq 1$. Consequently, $d_k<d_{k+1}$ holds for all $k$ 
and \eqref{eqn:claim0} is proved. But this gives the existence of infinitely 
many distinct critical levels, contradicting the assumption that $\cA$ is finite.

Hence, the functional $J$ has infinitely many geometrically distinct critical points, and by
\cite[Lemma 4.3]{evequoz-weth-dual} and \cite[Lemma 2.4]{e-helm-2d}, 
these give rise to geometrically distinct strong solutions
of \eqref{eq:33b}. This concludes the proof.
\end{proof}

\section{The asymptotically periodic problem}\label{sec:asympt_periodic}
In this section, we give sufficient conditions for the existence of a (nontrivial) 
solution to \eqref{eq:33b} in the case where 
\begin{equation}\label{eqn:asympt_per}
\lim_{R\to\infty}\operatorname*{ess\ sup}_{|x|\geq R}|Q(x)-Q_\infty(x)|=0,
\end{equation}
for some bounded, nonnegative and $\Z^N$-periodic function $Q_\infty\not\equiv 0$.

In order to show the existence of a critical point for $J$, we shall compare its behavior with that of the limit energy functional:
\begin{equation*}
J_\infty(v)=\frac1{p'}\int_{\R^N}|v|^{p'}\, dx - \frac12 \int_{\R^N}Q_\infty^\frac1p v\mR(Q_\infty^\frac1p v)\, dx, \quad v\in L^{p'}(\R^N).
\end{equation*}
Considering the minimax levels
$$
c=\inf\limits_{v\neq 0}\ \sup\limits_{t>0}J(tv)\quad\text{ and }\quad
c_\infty=\inf\limits_{v\neq 0}\ \sup\limits_{t>0}J_\infty(tv),
$$
we first remark that $0<c, c_\infty<\infty$, as follows from Lemma~\ref{lem:MP}. 
In addition, we notice that for each $v\in U^+:=\left\{ u\in L^{p'}(\R^N)\, :\, \int_{\R^N}Q^\frac1pu\mR(Q^\frac1pu)\, dx>0\right\}$
there is a unique $t=t_v>0$, given by
\begin{equation}\label{eqn:t_v}
t_v^{2-p'}=\frac{\int_{\R^N}|v|^{p'}\, dx}{\int_{\R^N}Q^\frac1pv\mR(Q^\frac1pv)\, dx},
\end{equation}
for which $J(t_vv)>J(sv)$ holds for all $s>0$, $s\neq t_v$. We can therefore write
\begin{equation}\label{eqn:c_inf_fiber}
c=\inf\limits_{v\in U^+}J(t_vv)=\inf\limits_{v\in U^+}\left(\frac1{p'}-\frac12\right)t_v^{p'}\int_{\R^N}|v|^{p'}\, dx,
\end{equation}
since $\sup\limits_{t>0}J(tv)=\infty$ for $v\notin U^+\cup\{0\}$. Similarly,
\begin{equation}
c_\infty=\inf\limits_{v\in U_\infty^+}J_\infty(t^\infty_vv)
=\inf\limits_{v\in U_\infty^+}\left(\frac1{p'}-\frac12\right)(t^\infty_v)^{p'}\int_{\R^N}|v|^{p'}\, dx,
\end{equation}
with corresponding definitions for $U_\infty^+$ and $t_v^\infty$.

\begin{lemma}\label{lem:c_b}
Let $2_\ast\leq p< 2^\ast$ and $Q, Q_\infty\in L^\infty(\R^N)\backslash\{0\}$ be nonnegative functions.
Consider the mountain-pass level
$$
b=\inf\limits_{\gamma\in\Gamma}\max\limits_{t\in[0,1]}J(\gamma(t)), 
\quad\text{ where }\Gamma=\left\{\gamma\in C([0,1],L^{p'}(\R^N))\, :\, \gamma(0)=0\text{ and }J(\gamma(1))<0\right\}.
$$
Then
\begin{itemize}
\item[(i)] $c=b$.
\item[(ii)] If \eqref{eqn:asympt_per} holds, then $c\leq c_\infty$.
\end{itemize}
\end{lemma}
\begin{proof}
(i) For each $v\in U^+$, we have $\lim\limits_{t\to\infty}J(tv)=-\infty$. Using \eqref{eqn:c_inf_fiber}, 
we obtain $c=\inf\limits_{v\in U^+}\ \sup\limits_{t>0}J(tv)\geq b$.
For the converse inequality, we note that since $p'<2$ there exists $\eta>0$ such that 
$$
J'(v)v=\int_{\R^N}|v|^{p'}\, dx -\int_{\R^N}Q^\frac1pv\mR(Q^\frac1pv)\, dx>0,\quad\text{ for all }0<\|v\|_{p'}\leq \eta.
$$
Also, if $\gamma\in\Gamma$, then $0>J(\gamma(1))
=\frac1{p'}\int_{\R^N}|\gamma(1)|^{p'}\, dx-\frac12\int_{\R^N}Q^\frac1p\gamma(1)\mR(Q^\frac1p\gamma(1))\, dx$
and this gives $J'(\gamma(1))\gamma(1)<0$. Since $\gamma(0)=0$ and $\|\gamma(1)\|_{p'}>\eta$, 
there is some $0<\widetilde{t}<1$ such that with $v=\gamma(\,\widetilde{t})$,
there holds $J'(v)v=0$ and $v\neq 0$. In particular, $v\in U^+$ and $t_v=1$, which implies that
$$
\max\limits_{t\in[0,1]}J(\gamma(t))\geq J(v)=J(t_vv)\geq c,
$$ 
and since $\gamma\in \Gamma$ was chosen arbitrarily,
we obtain $b\geq c$.

(ii) As a consequence of \eqref{eqn:asympt_per}, the following holds
for every $v\in L^{p'}(\R^N)$ and every sequence $(y_n)_n\subset\R^N$ such that $\lim\limits_{n\to\infty}|y_n|=\infty$:
\begin{align*}
\lim_{n\to\infty}\int_{\R^N}Q^\frac1p v(\cdot-y_n)\mR(Q^\frac1pv(\cdot-y_n))\, dx
&=\lim_{n\to\infty}\int_{\R^N}Q^\frac1p(\cdot+y_n) v\mR(Q^\frac1p(\cdot+y_n)v)\, dx\\
&=\int_{\R^N}Q_\infty^\frac1p v\mR(Q_\infty^\frac1pv)\, dx.
\end{align*}
Consider $v\in U_\infty^+$ and let $v_n=v(\cdot-y_n)$ for some sequence $(y_n)_n\subset\R^N$
such that $\lim\limits_{n\to\infty}|y_n|=\infty$. Then for $n$ large enough, the above property gives 
$v_n\in U^+$ and from \eqref{eqn:t_v}, $\lim\limits_{n\to\infty}t_{v_n}=t_v^\infty$. Consequently,
$$
c\leq \lim_{n\to\infty}\left(\frac1{p'}-\frac12\right)t_{v_n}^{p'}\int_{\R^N}|v_n|^{p'}\, dx
=\left(\frac1{p'}-\frac12\right)(t^\infty_v)^{p'}\int_{\R^N}|v|^{p'}\, dx,
$$
and since $v\in U_\infty^+$ was arbitrarily chosen, we conclude that $c\leq c_\infty$.
\end{proof}

\begin{proposition}\label{prop:c<c_inf}
Let $2_\ast\leq p<2^\ast$, $Q, Q_\infty\in L^\infty(\R^N)\backslash\{0\}$ be nonnegative functions 
such that \eqref{eqn:asympt_per} holds.
If $c<c_\infty$, then $J$ has a (nontrivial) critical point at level $c$.
\end{proposition}
\begin{proof}
Since $c=b$ and $b$ is the mountain-pass level associated to $J$, a standard deformation argument 
(see e.g. \cite[Lemma 6.1]{evequoz-weth-dual}) gives
the existence of a Palais-Smale sequence $(v_n)_n\subset L^{p'}(\R^N)$ for $J$ at level $c$. 
According to Lemma~\ref{lem:PS_sequences}, we can assume
(going to a subsequence) that $v_n\weakto v$ weakly in $L^{p'}(\R^N)$ where $J'(v)=0$ and 
$J(v)\leq\liminf\limits_{n\to\infty}J(v_n)=c$.

If $v\neq 0$, then we are done, since $J'(v)=0$ gives $v\in U^+$ and $t_v=1$, from which $J(v)\geq c$ follows.
Assuming by contradiction that $v_n\weakto 0$ holds, we write
\begin{align*}
\int_{\R^N}Q_\infty^\frac1pv_n\mR(Q_\infty^\frac1pv_n)\, dx
 &=\int_{\R^N}Q^\frac1pv_n\mR(Q^\frac1pv_n)\, dx\\
 &\qquad +\int_{\R^N}(Q_\infty^\frac1p-Q^\frac1p)v_n\mR\bigl([Q_\infty^\frac1p+Q^\frac1p]v_n\bigr)\, dx
\end{align*}
and see that for $r>0$,
\begin{align*}
&\left|\int_{\R^N}(Q_\infty^\frac1p-Q^\frac1p)v_n\mR\bigl([Q_\infty^\frac1p+Q^\frac1p]v_n\bigr)\, dx\right|\\
&\qquad\leq\Bigl(\|Q_\infty\|_\infty^\frac1p+\|Q\|_\infty^\frac1p\Bigr)
\Bigl[C\|v_n\|_{p'}^2\operatorname*{ess\ sup}\limits_{|x|>r}\bigl|Q_\infty(x)^\frac1p-Q(x)^\frac1p\bigr|\\
&\qquad\qquad +\|v_n\|_{p'}\|1_{B_r(0)}\mR\bigl([Q_\infty^\frac1p+Q^\frac1p]v_n\bigr)\|_p\Bigr].
\end{align*}
Since by assumption \eqref{eqn:asympt_per} we have 
$\operatorname*{ess\ sup}\limits_{|x|>r}|Q_\infty(x)^\frac1p-Q(x)^\frac1p|\to 0$, as $r\to\infty$, 
and since the operator $1_{B_r(0)}\mR$: $L^{p'}(\R^N)$ $\to$ $L^p(\R^N)$
is compact for every $r>0$ (see \cite[Lemma 4.1]{evequoz-weth-dual}), we conclude that
\begin{align*}
\lim_{n\to\infty}\int_{\R^N}Q_\infty^\frac1pv_n\mR(Q_\infty^\frac1pv_n)\, dx
&=\lim_{n\to\infty}\int_{\R^N}Q^\frac1pv_n\mR(Q^\frac1pv_n)\, dx\\
&=\left(\frac1{p'}-\frac12\right)^{-1}\lim_{n\to\infty}[J(v_n)-\frac1{p'}J'(v_n)v_n]=\frac{2p'c}{2-p'}>0.
\end{align*}
Hence, for $n$ large enough, $v_n\in U_\infty^+$ and there holds
\begin{align*}
(t^\infty_{v_n})^{2-p'}=\frac{\int_{\R^N}|v_n|^{p'}\, dx}{\int_{\R^N}Q_\infty^\frac1pv_n\mR(Q_\infty^\frac1pv_n)\, dx}
=\frac{(\frac1{p'}-\frac12)^{-1}[J(v_n)-\frac12J'(v_n)v_n]}{(\frac1{p'}-\frac12)^{-1}[J(v_n)-\frac1{p'}J'(v_n)v_n]+o(1)}\to 1, 
\end{align*}
as $n\to\infty$. As a consequence, we find
\begin{align*}
c_\infty\leq \lim_{n\to\infty}J_\infty(t_{v_n}^\infty v_n)
=\left(\frac1{p'}-\frac12\right)\lim_{n\to\infty}(t_{v_n}^\infty)^2\int_{\R^N}Q_\infty^\frac1pv_n\mR(Q_\infty^\frac1pv_n)\, dx
=c,
\end{align*}
which contradicts the assumption $c<c_\infty$ and ends the proof.
\end{proof}

We now restrict to the case $p>2_\ast$ in which the nonvanishing property (Theorem~\ref{thm:nonvanishing}) 
holds for the resolvent $\mR$.

In this case, we note that the level $c_\infty$ corresponding to the periodic functional $J_\infty$ is attained and coincides 
with the least-energy level, i.e.,
$$
c_\infty=\inf\{J_\infty(v)\, :\, v\in L^{p'}(\R^N), \ v\neq 0\text{ and }J_\infty'(v)=0\}.
$$
Indeed, using the weak lower semicontinuity of $J_\infty$ along Palais-Smale sequences (see Lemma \ref{lem:PS_sequences}), 
we obtain that the nontrivial critical point $w$ of $J_\infty$ given by the proof of \cite[Theorem 6.2]{evequoz-weth-dual} has no 
higher energy than the mountain-pass level $b_\infty$. Since by Lemma~\ref{lem:c_b}(i), 
$c_\infty=b_\infty$, and since every critical point $v$ of $J_\infty$ satisfies $J_\infty(v)\geq c_\infty$, 
we find that $J_\infty(w)=c_\infty$.

As a consequence, we obtain the following existence result.
\begin{theorem}\label{thm:asympt_per}
Let $2_\ast<p<2^\ast$, $Q, Q_\infty\in L^\infty(\R^N)\backslash\{0\}$ be nonnegative functions such that $Q_\infty$ is $\Z^N$-periodic.
If $Q_\infty\leq Q$ a.e. on $\R^N$
and \eqref{eqn:asympt_per} is satisfied,
then $J$ has a nontrivial critical point at level $c$ and there holds $c=\inf\{J(v)\, :\, v\in L^{p'}(\R^N), \ v\neq 0\text{ and }J'(v)=0\}$.
\end{theorem}
\begin{proof}
If $c<c_\infty$, the conclusion follows from Proposition~\ref{prop:c<c_inf}. Assume now that $c=c_\infty$.
Then, according to the above remark, there exists $w\in L^{p'}(\R^N)$ such that $J_\infty'(w)=0$ and $J_\infty(w)=c_\infty$. 
Since $0\leq Q_\infty\leq Q$ a.e. in $\R^N$, by assumption, we may consider
$v:=\left(\frac{Q_\infty}{Q}\right)^\frac1p w\in L^{p'}(\R^N)$. There holds
$$
\int_{\R^N}Q^\frac1pv\mR(Q^\frac1pv)\, dx = \int_{\R^N}Q_\infty^\frac1pw\mR(Q_\infty^\frac1pw)\, dx=\int_{\R^N}|w|^{p'}\, dx>0,
$$
and this gives $v\in U^+$. 
As a consequence, we find
\begin{align*}
c\leq J(t_vv)
&=\frac{t_v^{p'}}{p'}\int_{\R^N}\left(\frac{Q_\infty}{Q}\right)^{p'-1} |w|^{p'}\, dx-\frac{t_v^2}2
\int_{\R^N}Q_\infty^\frac1pw\mR(Q_\infty^\frac1pw)\, dx\\
&\leq J_\infty(t_vw)\leq J_\infty(w)=c_\infty,
\end{align*}
and since $c=c_\infty$ by assumption, we conclude that $J_\infty(t_vw)=J_\infty(w)$, and 
therefore $t_v=t^\infty_w=1$. Hence, $J(v)=c$ and $J'(v)v=0$.

In order to show that $J'(v)=0$, we proceed as follows: Consider $\varphi\in L^{p'}(\R^N)$.
Since $v\in U^+$ and $U^+$ is an open set, there is $\delta>0$
such that $v+s\varphi\in U^+$ for all $|s|<\delta$. 
Using the characterization \eqref{eqn:c_inf_fiber} of $c$, 
and since $J(v)=c$, $t_v=1$, 
we can write for $|s|<\delta$:
\begin{align*}
0\leq J\bigl(t_s(v+s\varphi)\bigr)-J(v)
&\leq J\bigl(t_s(v+s\varphi)\bigr)-J\bigl(t_sv\bigr)=s t_sJ'\bigl(t_s(v+s\tau\varphi)\bigr)\varphi,
\end{align*}
where $t_s=t_{v+s\varphi}$ and where $0\leq\tau=\tau(s)\leq1$ is given by the mean-value theorem. It follows that
\begin{align*}
J'\bigl(t_{-s}(v-s\tau\varphi)\bigr)\varphi\leq 0
\leq J'\bigl(t_s(v+s\tau\varphi)\bigr)\varphi
\end{align*}
for all $0<s<\delta$. Letting $s\to 0$ we obtain $J'(v)\varphi=0$,
since the functions $s\mapsto t_{v\pm s\varphi}$ are continuous on $(-\delta,\delta)$ and $t_v=1$.
Therefore, $v$ is a nontrivial critical point for $J$ at level $c$.
\end{proof}

\begin{remark}
The preceding result gives conditions ensuring the existence of a nontrivial critical point for $J$, and one might wonder
whether or not, similar to the periodic case, the existence of infinitely many solutions can be proven
under the assumption \eqref{eqn:asympt_per}. 

It is a fact that the dual functional $J$ associated to \eqref{eq:33b} bears many similarities with
the one considered by Alama and Li in \cite{alama-li92,alama-li92b} and associated to a NLS equation with frequency
in a spectral gap. In particular, a splitting property holds for Palais-Smale sequences of the dual functional, and the arguments 
in \cite{alama-li92b} can be used with only minor changes to give the existence of infinitely many critical points for $J$, assuming that \eqref{eqn:asympt_per} holds and that
$$
\text{there exist }0<a<\frac{c_\infty}{2}\text{ and }v_0\in \cK_\infty^c\text{ such that }v_0
\text{ is isolated in }\cK_\infty^{c+a}.
$$
It is however not clear whether this last assumption can be verified for some class of periodic potentials $Q_\infty$.

\end{remark}

\section*{Acknowledgements}
This research is supported by the Grant WE 2821/5-1 of the Deutsche Forschungsgemeinschaft (DFG). The author
would like to thank Tobias Weth for pointing out the problem studied here and for valuable advice.

\bibliographystyle{abbrv}
\bibliography{literatur.bib}

\end{document}